\documentclass{amsart}

\usepackage{amssymb}
\usepackage{color}

\input prepictex   \input pictex    \input postpictex

\definecolor{darkgreen}{rgb}{0,0.5,0}

\def\arr#1#2{\arrow <2mm> [0.25,0.75] from #1 to #2}
\def\ssize{\scriptstyle}

\DeclareMathOperator\rad{{\rm rad}}

\DeclareMathOperator\Hom{{\rm Hom}}

\DeclareMathOperator\Ker{{\rm Ker}}

\renewcommand\Im{{\rm Im}}
\renewcommand\mod{{\rm mod}}

\newcommand\onto{\twoheadrightarrow}

\newcommand\bd{{\mathbf d}}
\newcommand\be{{\mathbf e}}
\newcommand\bh{{\mathbf h}}

\newtheoremstyle{mytheorems}{9pt}{6pt}{\itshape}{0pt}{\sc}{.}{ }{}
\newtheoremstyle{myremarks}{6pt}{3pt}{\normalfont}{0pt}{\it}{.}{ }{}

\theoremstyle{mytheorems}
\newtheorem{theorem}{Theorem}[section]
\newtheorem{lemma}[theorem]{Lemma}
\newtheorem{corollary}[theorem]{Corollary}
\newtheorem{proposition}[theorem]{Proposition}

\theoremstyle{myremarks}

\begin{document}
%
%{\tt\footnotesize[\ver, \today]}
%\bigskip\bigskip
%
\title[Linear Operators with Two Invariant Subspaces]{The Swiss Cheese
       Theorem for Linear Operators with Two Invariant Subspaces}
 % invariant subspaces}
\author{Audrey Moore}
% \address{Department of Mathematical Sciences\\ Delaware State University\\ 
%   1200 N DuPont Highway\\ Dover, DE 19901}
\email{audreydoughty@yahoo.com}
\author{Markus Schmidmeier}
\thanks{This research is partially supported 
by a~Travel and Collaboration Grant from the Simons Foundation (Grant number
245848 to the second named author).}
% \address{Department of Mathematical Sciences\\ Florida Atlantic University\\ 
%   777 Glades Road\\ Boca Raton, FL 33431}
\email{markus@math.fau.edu}

\begin{abstract}\sloppy
We study systems  $(V,T,U_1,U_2)$ consisting of a 
  finite dimensional vector space $V$, a nilpotent
  $k$-linear operator $T:V\to V$ and two $T$-invariant
  subspaces $U_1\subset U_2\subset V$.
  Let $\mathcal S(n)$ be the category of such systems where the operator $T$ acts
  with nilpotency index at most $n$.
  We determine the dimension types $(\dim U_1, \dim U_2/U_1, \dim V/U_2)$ 
  of indecomposable systems in $\mathcal S(n)$ for $n\leq 4$. 
  It turns out that in the case where $n=4$ there are infinitely
  many such triples $(x,y,z)$, they all lie in the 
  cylinder given by $|x-y|,|y-z|,|z-x|\leq 4$.
  But not each dimension type in the cylinder can be realized by an indecomposable system.
  In particular, there are holes in the cylinder.  Namely,
  no triple in $(x,y,z)\in (3,1,3)+\mathbb N(2,2,2)$ can be realized, 
  while each neighbor
  $(x\pm1,y,z), (x,y\pm1,z),(x,y,z\pm1)$ can.
  Compare this with Bongartz' No-Gap Theorem, which states that for an associative algebra $A$
  over an algebraically closed field, there is no gap in the lengths of the indecomposable
  $A$-modules of finite dimension.
\end{abstract}

\maketitle

\section{Introduction}
%---------------------

Let $k$ be a field.
We are interested in all possible configurations consisting of 
a finite dimensional $k$-vector space $V$, a linear operator  $T$ acting on $V$,
and a pair of subspaces $U_1\subset U_2$ of $V$ which are invariant under the
action of $T$.
We call such configurations $(V,T,U_1, U_2)$ {\it systems} or {\it quadruples,} 
they form the objects
of a category $\mathcal S$; the morphisms in $\mathcal S$ 
from $(V,T,U_1,U_2)$ to $(V',T',U_1',U_2')$
are those linear maps $f:V\to V'$ which commute with the action of the operator
in the sense that $fT'=Tf$ holds, and which preserve the subspaces, that is,
$f(U_1)\subset U_1'$ and $f(U_2)\subset U_2'$.

\subsection{The categories $\mathcal S(n)$}
%------------------------------------------

The category $\mathcal S$ is additive and 
has the Krull-Remak-Schmidt property, so every system has 
a unique direct sum decomposition into indecomposable ones.  
Moreover, $\mathcal S$ carries the exact structure given 
by sequences of systems and homomorphisms which give rise to three
short exact sequences of vector spaces, namely of the ambient spaces,
the small, and the intermediate subspaces.  There are three simple objects,
$$E_1=(k,0,k,k),\quad E_2=(k,0,0,k), \quad E_3=(k,0,0,0),$$
in each the ambient space $V$ is one-dimensional, and either both,
one, or none of the subspaces are equal to the ambient space.
The {\it dimension type} of a system $M=(V,T,U_1,U_2)$ 
is the triple $$(x,y,z)=(\dim U_1, \dim U_2/U_1, \dim V/U_2)$$
which consists of the multiplicities of $E_1$, $E_2$, $E_3$
as composition factors of $M$, respectively.

\medskip
For $n$ a natural number, let $\mathcal S(n)$ 
be the full subcategory of $\mathcal S$ of all
systems $(V,T,U_1,U_2)$ where the operator $T:V\to V$ 
acts with nilpotency index at most $n$, so
$T^n=0$ holds. 
Categories of type $\mathcal S(n)$ have been studied in the first
author's doctoral dissertation \cite{diss}.  
In this paper, we are interested in the dimension types 
of the indecomposable objects in $\mathcal S(n)$.

\subsection{The hexagonal picture}
%---------------------------------

For each of the categories  $\mathcal S(n)$ 
where $n\leq 4$ we will determine all dimension types of 
indecomposable systems.  The representation finite cases 
are dealt with in Section~\ref{section-finite};
the categories 
$\mathcal S(1)$, $\mathcal S(2)$ and $\mathcal S(3)$
have 3, 9, 27 indecomposable objects, respectively.  
By contrast, the classification of indecomposable systems 
in $\mathcal S(n)$ for $n>4$ is considered an infeasible problem.
So the category $\mathcal S(4)$ is of particular interest as
a borderline case.  
Our main result is the description of the possible dimension types 
for the category $\mathcal S(4)$:

\begin{theorem}\label{theorem-dimension-triples}
A triple $(x,y,z)\neq 0$ of non-negative integers is the dimension type
of an indecomposable system in $\mathcal S(4)$ 
if and only if either all coordinate differences
$|x-y|$, $|y-z|$, $|z-x|$ are all at most 3 and 
$$(x,y,z)\notin \{(1,3,1),(1,4,1),(3,1,3)\}\quad+\mathbb N(2,2,2)$$
or else
\begin{align*} (x,y,z) \in \{ & (4,0,0),(0,4,0),(0,0,4), \\
           &        (4,2,0),(2,4,0),(4,0,2),(2,0,4),(0,4,2),(0,2,4),\\
           &        (5,4,1),(5,2,1),(1,4,5),(1,2,5)\}\quad+\mathbb N(2,2,2). 
\end{align*}
\end{theorem}

In particular, the triple $(3,1,3)$ is not the dimension type of any
indecomposable object in $\mathcal S(4)$, but each neighbor
$(3\pm1,1,3), (3,1\pm1,3), (3,1,3\pm1)$ is. Hence the title of the paper.

\smallskip
Note that a non-zero triple $(x,y,z)\in\mathbb N^3$ 
occurs as the dimension type
of an indecomposable object in $\mathcal S(4)$ if and only if
the triple $(x+2,y+2,z+2)$ does.
We observe that for each dimension type, all coordinate differences
are at most 4, hence it lies in the first octant part of a 
cylinder with axis $(2,2,2)$.

\medskip
We project this cylinder along its axis
onto the plane given by the equation $x+y+z=0$.  
Up to a shift by a multiple of $(2,2,2)$, 
each dimension type $(x,y,z)$ either (a) lies on one of the coordinate planes
and corresponds to a point in the diagram on the left, 
or (b) has minimum entry 1 and corresponds to a point in the diagram
on the right, or (c) is a positive integer multiple of $(2,2,2)$.

%
%      min(x,y,z) even
%
$$
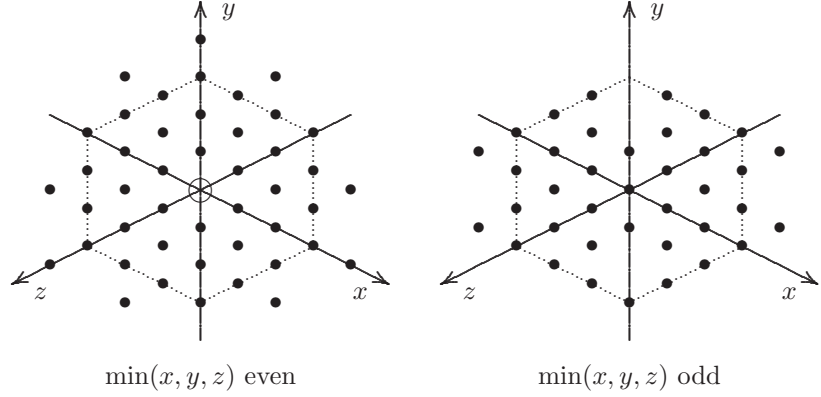
\beginpicture\setcoordinatesystem units <2.5mm,2.5mm>
\put {} at -10 -10
\put {} at 10 10 
\put{$\min(x,y,z)$ even} at 0 -10
\arr{0 -8}{0 10}
\arr{8 4}{-10 -5}
\arr{-8 4}{10 -5}
\put{$x$} at 8.5 -5.5
\put{$y$} at 1.5 9.5
\put{$z$} at -8.5 -5.5
\setdots<2pt>
\plot 0 6  -6 3  -6 -3  0 -6  6 -3  6 3  0 6 /
\multiput{$\bullet$} at -4 -6  0 -6  4 -6  -2 -5  2 -5  
   -8 -4  -4 -4  0 -4  4 -4  8 -4  -6 -3  -2 -3  2 -3  6 -3
   -4 -2  0 -2  4 -2  -6 -1  -2 -1  2 -1  6 -1  -8 0  -4 0  4 0  8 0
   -6 1  -2 1  2 1  6 1  -4 2  0 2  4 2  -6 3  -2 3  2 3  6 3  
   -4 4  0 4  4 4  -2 5  2 5  -4 6  0 6  4 6  0 8 /
\put{$\bigcirc$} at 0 0 
\endpicture
\qquad
%
%          min(x,y,z) odd
%
\beginpicture\setcoordinatesystem units <2.5mm,2.5mm>
\put {} at -10 -10
\put {} at 10 10 
\put{$\min(x,y,z)$ odd} at 0 -10
\arr{0 -8}{0 10}
\arr{8 4}{-10 -5}
\arr{-8 4}{10 -5}
\put{$x$} at 8.5 -5.5
\put{$y$} at 1.5 9.5
\put{$z$} at -8.5 -5.5
\setlinear
\setdots<2pt>
\plot 0 6  -6 3  -6 -3  0 -6  6 -3  6 3  0 6 /
\multiput{$\bullet$} at  0 -6   -2 -5  2 -5  
   -4 -4  4 -4  -6 -3  -2 -3  2 -3  6 -3
   -8 -2  -4 -2  0 -2  4 -2  8 -2  -6 -1  -2 -1  2 -1  6 -1  -4 0  0 0  4 0
   -6 1  -2 1  2 1  6 1  -8 2  -4 2  0 2  4 2  8 2  -6 3  -2 3  2 3  6 3  
   -4 4  4 4  -2 5  2 5  /
\endpicture
$$

\subsection{No gaps but holes}
%-----------------------------

Suppose $\mathcal C$ is an exact category such that the Grothendieck
group has basis given by the classes of $m$ simple objects.
The {\it dimension type} of an object in $\mathcal C$ is the 
corresponding element in $\mathbb Z^m$.
We consider the set $\mathcal V$ of all dimension types of 
indecomposable objects.
By $B(r)$ we denote the (closed) 
ball in $\mathbb R^m$ with center 0 and radius $r$.
We are interested in topological properties of $\mathcal V$, or rather in 
topological properties of the region $\mathcal V+B(r)$ in $\mathbb R^m$,
for some suitable radius $r$.

\medskip{\it Is $\mathcal V$ connected?}  

More precisely, we ask if the region
$\mathcal V+B(\frac12)$ is connected.

\smallskip
In \cite[p.\ 655]{ringel2011}, Ringel states the following open problem.
Let $A$ be a finite dimensional
algebra over an algebraically closed field.  Given an indecomposable $A$-module
of length $n>1$, is there an indecomposable submodule or factor module of dimension
$n-1$?

\smallskip
We note that a positive answer to Ringel's question yields that the set
$\mathcal V$ of dimension types of indecomposable $A$-modules is connected,
provided only that the algebra $A$ is connected.

\smallskip
For modules over a finite dimensional algebra $A$ over an algebraically closed field,
the No Gap Theorem states that whenever there is an indecomposable $A$-module of
length $n>1$, then there is one of length $n-1$.
Note that the connectedness of $\mathcal V$ implies the statement in the No 
Gap Theorem \cite[Theorem~1]{bongartz2013}.

\medskip{\it Is $\mathcal V$ simply connected?}

Here we are interested in the
region $\mathcal V+B(\frac12\sqrt2)$ in $\mathbb R^m$.

\smallskip
Note that $\mathcal V$ is not always simply connected. 
For some algebras, for example $kQ/(\alpha^2)$
where $Q$ is a cycle of length at least 3 with all arrows labelled $\alpha$, 
the set $\mathcal V$ has a hole at the origin. 

\smallskip
We present here an example from the theory of linear operators with one
invariant subspace.
The category $\mathcal S_1(5)$ 
studied in \cite[(6.5)]{rs} consists of 
all triples $(V,T,U)$ where $V$ is a finite dimensional vector space, $T:V\to V$
a linear operator acting nilpotently with nilpotency index at most 5, and $U$
a subspace of $V$ invariant under $T$.  
There are two simple objects, $(k,0,k)$ and $(k,0,0)$; 
their multiplicities as composition factors of an object $(V,T,U)$
define the dimension pair $(x,y)$ where $x=\dim U$, $y=\dim V/U$.
In this example, the set $\mathcal V$ consists of the following 50 points.
$$\hbox{\beginpicture
\setcoordinatesystem units <.5cm,.5cm>
\arr{0 -1}{0 8}
\arr{-1 0}{8 0}
\put{$x$} at  0.5 8.2
\put{$y$} at  8.2 -.5
\setdashes <.5mm> 
\setplotarea x from 0 to 7, y from 0 to 7 
\grid {7} {7}
\plot -1 -1  7.5 7.5 /
% \plot -1 4.5   2 7.5  /
% \plot 4.5 -1  7.5 2 /
\setsolid
\plot 5 -0.5  5 0.2 /
\put{$\ssize 5$} at 4.7 -.5
\plot -0.2 5  0.2 5 /
\put{$\ssize 5$} at -.5 4.75 
\put{$\ssize 0$} at -.5 .3 
\multiput{$\ssize \bullet$} at 1 0  2 0  3 0  4 0  5 0  
                             0 1  1 1  2 1  3 1  4 1
                             0 2  1 2  2 2  3 2  0 3  1 3  2 3  0 4  1 4  0 5  
                             2 2  3 2  4 2  2 3  3 3  2 4  3 4
	                     5 2  6 2  4 3  5 3  2 4  4 4  2 5  3 5  2 6 
                             6 3  5 4  4 5  3 6  6 4  4 6  6 5  5 6  6 6 /
\multiput{$\ssize 1$} at 1.2 .3   2.2 .3  3.2 .3   4.2 .3   5.2 .3 
                      .2 1.3  1.2 1.4  2.2 1.3  3.2 1.3  4.2 1.3  
                    .2 2.3  1.2 2.3   .2 3.3  1.2 3.3  .2 4.3  1.2 4.3  .5 5.25 /
\multiput{$\ssize 2$} at 5.2 2.3  6.2 2.3  5.2 3.3  
                        2.2 5.3  3.2 5.3  2.2 6.3 /
\multiput{$\ssize 12$} at 3.4 2.3  2.4 3.3  2.3 2.5 /
\multiput{$\ssize 22$} at 4.4 2.3  3.3 3.5  4.4 3.3  2.4 4.3  3.4 4.3  4.3 4.5 /
\multiput{$\ssize 3$} at 6.2 3.3  5.2 4.3  4.2 5.3  3.2 6.3  6.2 4.3  4.2 6.3
	                 6.2 5.3  5.2 6.3  6.2 6.4 /
\endpicture}
$$
(For each indecomposable system, the number next to the vertex counts the 
Jordan blocks of the linear operator $(V,T)$.)

\smallskip
In this example, $\mathcal V$ is connected but not simply connected since the 
vertex $(5,5)$ is missing.

\medskip{\it Is it possible that $\mathcal V$ has holes?}

Here we consider the region $\mathcal V+B(\frac13\sqrt6)$ in $\mathbb R^m$.
The radius is such that the octahedron with vertex set $\{(\pm1,0,0),(0,\pm1,0),(0,0,\pm1)\}$
is simply connected and has a hole at the origin.

\medskip
We deduce from Theorem~\ref{theorem-dimension-triples} that
the set $\mathcal V$ for the category $\mathcal S(4)$ is connected and simply
connected but has holes in positions $(3,1,3)$, $(5,3,5)$,
$(7,5,7)$, etc.

\subsection{Linear operators in control theory}
%-------------------------------------------------

We would like to point out that 
linear operators with two invariant subspaces occur naturally
in the theory of linear time-invariant dynamical systems.
Such a system $\Sigma$ consists of the following 
first order differential equations.
$$\Sigma:\quad \left\{\begin{array}{rl}\dot x(t) & = B\, x(t)+A\, u(t)\\ y(t) &= C\, x(t)
\end{array}\right.$$
Here, $x(t)\in V$ is the {\it state}, $u(t)\in U$ the {\it input} or {\it control},
and $y(t)\in W$ the {\it output} at time $t$.  The term {\it time invariance}
refers to the linear maps $A$, $B$ and $C$. 
Note that the data in $\Sigma$ define a representation of the following quiver.

$$
\hbox{\beginpicture
\setcoordinatesystem units <1cm,1cm>
%==========================================
\put{} at 0 0
\put{} at 4 1.2
%\multiput{$\circ$} at 0 0  2 0  4 0 /
\arr{0.4 0}{1.6 0}
\arr{2.4 0}{3.6 0}

\put{$U$} at  0 0
\put{$V$} at  2 0
\put{$W$} at  4 0

\circulararc 300 degrees from 2.3 .1 center at 2 .6
\arr{2.4 .18}{2.3 .1}
\put{$\ssize B$} at 2.9 .6
\put{$\ssize A$} at 1 -.3
\put{$\ssize C$} at 3 -.3
\endpicture}
$$

\medskip
A system $\Sigma$ is {\it completely controllable} if any state can be reached
from the zero state in finite time, for some piecewise continuous input
function $u(t)$. Dually, 
$\Sigma$ is {\it completely observable} if any state can be 
uniquely determined from the values of $y(t)$, taken over a finite time 
interval.
The controllable subspace $U_2$ of $V$ is the sum of the images
$U_2=\sum_{k\geq 0}\Im B^kA$ \cite[Theorem~1.24]{ks}, 
while the non-observable subspace is given by 
the intersection of the kernels $U_1'=\bigcap_{k\geq0}\Ker CB^k$
\cite[Theorem~1.33]{ks}.  Clearly, both $U_2$ and $U_1'$ are invariant under
multiplication by $B$. Putting $T=B$ and $U_1=U_1'\cap U_2$, we obtain a 
quadruple $(V,T,U_1,U_2)$ of a linear operator and two invariant subspaces.
This system is in $\mathcal S(n)$ if all Jordan blocks for the operator $B$
have size at most $n$.

\medskip
The Kalman decomposition of a control system $\Sigma$ yields a system,
called the {\it minimal realization,}  
which is completely controllable and completely observable.
The state space of this system is obtained as the subquotient
$U_2/U_1$ of the subspaces in the quadruple  $(V,T,U_1,U_2)$.

$$
\hbox{\beginpicture
\setcoordinatesystem units <1cm,1cm>
%==========================================
\put{} at 0 0
\put{} at 4 1.2
%\multiput{$\circ$} at 0 0  2 0  4 0 /
\arr{0.4 0}{1.6 0}
\arr{2.4 0}{3.6 0}

\put{$U$} at  0 0
\put{$\displaystyle \frac{U_2\strut}{U_1\strut}$} at  2 0
\put{$W$} at  4 0

\circulararc 300 degrees from 2.3 .1 center at 2 .6
\arr{2.4 .18}{2.3 .1}
\put{$\ssize{\overline B}$} at 2.9 .6
\put{$\ssize{\overline A}$} at 1 -.3
\put{$\ssize{\overline C}$} at 3 -.3
\endpicture}
$$

\subsection{Related results}
%---------------------------
The problem of classifying the embeddings of an 
invariant subspace in a linear operator
can be traced back to the 1934 paper by Garrett Birkhoff \cite{B}
where he asks the corresponding question for embeddings of a subgroup
in an abelian group.
There have been many generalizations of this problem, in particular
in \cite{simson}, the representation types of categories of chains
of invariant subspaces (which includes our case of chains of length 2)
have been determined.
For homological properties of chain categories we refer the reader to
\cite{xzz}.  
In \cite{moore}, the Auslander and Ringel-Tachikawa Theorem
has been shown for chain categories.
The relation with singularity theory has been pointed out
in \cite{klm}.

\subsection{Organization of this paper}
%--------------------------------------
In Section~2 we review the description of the category $\mathcal S(4)$
as given in \cite{diss}.
In particular, the indecomposable objects in $\mathcal S(4)$ either correspondto certain modules over a tubular algebra $A$, or else occur on one of three rays 
that are inserted in the Auslander-Reiten quiver for $A$.

\smallskip
We compute the dimension types which can be realized by $A$-modules
in Section~3. 
Two directions are necessary:  First, we show that dimension types which are 
not in the list cannot occur since they are not roots of a certain integral 
quadratic form.  Second, the dimension types in the list will be realized by
modules over some tame domestic algebra.

\smallskip
In Section~4 we list the dimension types
for the indecomposable systems in the categories $\mathcal S(1)$,
$\mathcal S(2)$, $\mathcal S(3)$.

\section{The category $\mathcal S(4)$}
%-------------------------------------
\label{section-s4}

The main result is a consequence of the detailed investigation of
the category $\mathcal S(4)$ in \cite{diss}.
In this section we review some of the results.

\subsection{Quiver representations}
%----------------------------------
\label{section-quiver-representations}
We can consider the objects in $\mathcal S(4)$ as quadruples 
$(V,T,U_1,U_2)$, or alternatively, as nilpotent linear operators with
two invariant subspaces, one contained in the other. 
There is a third way to describe those objects, namely as representations
of the quiver
$$
\hbox{\beginpicture
\setcoordinatesystem units <.7cm,.7cm>
%==========================================
\put{} at 0 0
\put{} at 4 1.2
\put{$Q:$} at -3 0
\multiput{$\circ$} at 0 0  2 0  4 0 /
\arr{0.4 0}{1.6 0}
\arr{2.4 0}{3.6 0}
\circulararc 300 degrees from .3 .1 center at 0 .6
\arr{.4 .18}{.3 .1}
\circulararc 300 degrees from 2.3 .1 center at 2 .6
\arr{2.4 .18}{2.3 .1}
\circulararc 300 degrees from 4.3 .1 center at 4 .6
\arr{4.4 .18}{4.3 .1}

\put{$\ssize \beta'$} at 1 -.5
\put{$\ssize \beta$} at 3 -.5
\put{$\ssize \alpha''$} at .9 .8
\put{$\ssize \alpha'$} at 2.9 .8
\put{$\ssize \alpha$} at 4.9 .8
\endpicture}
$$
which satisfy the commutativity relations $\beta\alpha'=\alpha\beta$,
$\beta'\alpha''=\alpha'\beta'$; the nilpotency relation $\alpha^4=0$
and the additional condition that the maps representing $\beta$
and $\beta'$ are monomorphisms.

\subsection{Coverings}
%---------------------
We will also consider the following quiver which is the universal
covering for $Q$.
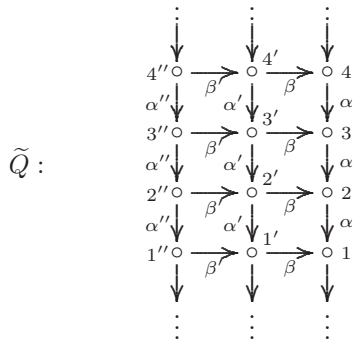
$$
\hbox{\beginpicture
\setcoordinatesystem units <0.5cm,0.4cm>
%==========================================
\put{} at -4 -2.2
\put{} at 4 8.4
\put{$\widetilde Q:$} at -4 3
\multiput{$\circ$} at 0 0  0 2  0 4  0 6  2 0  2 2  2 4  2 6  4 0  4 2  4 4  4 6 /
\arr{0 7.6}{0 6.4} \arr{0 5.6}{0 4.4} \arr{0 3.6}{0 2.4} \arr{0 1.6}{0 0.4} \arr{0 -.4}{0 -1.6}
\arr{2 7.6}{2 6.4} \arr{2 5.6}{2 4.4} \arr{2 3.6}{2 2.4} \arr{2 1.6}{2 0.4} \arr{2 -.4}{2 -1.6}
\arr{4 7.6}{4 6.4} \arr{4 5.6}{4 4.4} \arr{4 3.6}{4 2.4} \arr{4 1.6}{4 0.4} \arr{4 -.4}{4 -1.6}
\arr{0.4 6}{1.6 6} \arr{2.4 6}{3.6 6}
\arr{0.4 4}{1.6 4} \arr{2.4 4}{3.6 4}
\arr{0.4 2}{1.6 2} \arr{2.4 2}{3.6 2}
\arr{0.4 0}{1.6 0} \arr{2.4 0}{3.6 0}

\put{$\ssize 3''$} at  -.5 4
\put{$\ssize 2''$} at  -.5 2
\put{$\ssize 1''$} at  -.5 0
\put{$\ssize 4''$} at  -.5 6
\put{$\ssize 4'$} at  2.5 6.5
\put{$\ssize 3'$} at  2.5 4.5
\put{$\ssize 2'$} at  2.5 2.5
\put{$\ssize 1'$} at  2.5 0.5
\put{$\ssize 1$} at 4.5 0
\put{$\ssize 2$} at 4.5 2
\put{$\ssize 3$} at 4.5 4
\put{$\ssize 4$} at 4.5 6

\multiput{$\ssize \alpha$} at 4.5 1  4.5 3  4.5 5 /
\multiput{$\ssize \alpha'$} at 1.5 1  1.5 3  1.5 5 /
\multiput{$\ssize \alpha''$} at -.5 1  -.5 3  -.5 5 /
\multiput{$\ssize \beta$} at 3 5.5  3 3.5  3 1.5  3 -.5 /
\multiput{$\ssize \beta'$} at 1 5.5  1 3.5  1 1.5  1 -.5 /
\multiput{$\vdots$} at 0 8.4  2 8.4  4 8.4  0 -2.2  2 -2.2  4 -2.2 /
\endpicture}
$$

Denote by $\mathcal C(\widetilde 4)$ the category of all 
finite dimensional representations of the 
quiver $\widetilde Q$ subject to the commutativity relations
$\alpha\beta=\beta\alpha'$ and $\alpha'\beta'=\beta'\alpha''$ and 
the nilpotency relations $\alpha^4=0$, $\alpha'^4=0$, $\alpha''^4=0$.
% Thus, if $A_\infty^\infty$ denotes the double infinite linear quiver with all
% arrows labelled $\alpha$, then 
% $\mathcal C(\widetilde 4)$ is the category of pairs of composable maps between
% $kA_\infty^\infty/(\alpha^4)$-modules.

\medskip
The category $\mathcal C(\widetilde 4)$ will serve as a reference category,
we denote by $K_0(\mathcal C(\widetilde 4))$ or $K_0$ its Grothendieck group.
This is a free abelian group with basis elements $\be_z$ corresponding to
the vertices $z$ of $\widetilde Q$. For a representation $M$ of $\widetilde Q$, 
the corresponding element in $K_0$ is denoted by its dimension vector $\dim M$;
if $d=\dim M$ and if $z$ is a vertex in $\widetilde Q$, then 
the corresponding component $d_z$ of $d$ is just the dimension of the vector
space $M_z$ at position $z$.

\medskip
Let $\mathcal S(\widetilde 4)$ be the full subcategory of $\mathcal C(\widetilde 4)$
of all representations $M$ for which the maps $M_\beta$, $M_{\beta'}$ are all monic.
Thus, an object in $\mathcal S(\widetilde 4)$ is just a representation of
$kA_\infty^\infty/(\alpha^4)$, together with two subrepresentations, one contained in the other.

\medskip
The graded shift $[1]$ acts on $\mathcal C(\widetilde 4)$ and hence 
on $\mathcal S(\widetilde 4)$,
it maps a representation $M$ to the representation $M[1]$ which is given by the vector
spaces $M[1]_z=M_{z-1}$, $z$ a vertex in $\widetilde Q$, 
where we use the notation $i'-1=(i-1)'$, $i''-1=(i-1)''$ 
for $i\in \mathbb Z$.

\subsection{The tubular algebra $A$}
%-----------------------------------
Consider the path algebra $A$ of the quiver $Q_A$ modulo the 
relations as indicated:  One commutativity relation for the rectangle,
and three zero relations.
$$
\hbox{\beginpicture
\setcoordinatesystem units <0.5cm,0.5cm>
%==========================================
\put{} at -3 0
\put{} at 6 10
\put{$Q_A:$} at -3 4
\multiput{$\circ$} at 0 2  0 4  2 6  4 0  4 2  4 4  4 6  4 8  6 4 /
\arr{2.4 6}{3.6 6}
\arr{0.4 4}{3.6 4}
\arr{0.4 2}{3.6 2}
\arr{0 3.6}{0 2.4}
\arr{4 7.6}{4 6.4}
\arr{4 5.6}{4 4.4}
\arr{4 3.6}{4 2.4}
\arr{4 1.6}{4 .4}
\arr{4.4 4}{5.6 4}

\put{$\ssize 3''$} at  -.5 4
\put{$\ssize 2''$} at  -.5 2
\put{$\ssize 4'$} at  1.5 6

\put{$\ssize \bar3$} at 6.5 4
\put{$\ssize 1$} at 3.5 0
\put{$\ssize 2$} at 3.5 1.5
\put{$\ssize 3$} at 3.5 3.5
\put{$\ssize 4$} at 3.5 6.5
\put{$\ssize 5$} at 3.5 8

\setdots<2pt>
\plot .6 3.7  3.4 2.3 /
\plot 4.2 7.9  4.4 7.8  4.5 7.6  4.5 0.4  4.4 0.2  4.2 0.1 /
\plot .1 4.2  .2 4.4  .4 4.5  5.6 4.5  5.8 4.4  5.9 4.2 /
\plot 2.1 5.8  2.2 5.6  2.4 5.5  5.7 5.5  5.9 5.4  6 5.2  6 4.4 /
\endpicture}
\qquad
\hbox{\beginpicture
\setcoordinatesystem units <0.5cm,0.5cm>
%==========================================
\put{} at -3 0
\put{} at 6 10
\put{$Q_A^+:$} at -3 4
\multiput{$\circ$} at 0 2  0 4  0 6  2 6  2 8  4 0  4 2  4 4  4 6  4 8  4 10 
     6 4 /
\arr{2.4 6}{3.6 6}
\arr{0.4 4}{3.6 4}
\arr{0.4 2}{3.6 2}
\arr{0 3.6}{0 2.4}
\arr{4 7.6}{4 6.4}
\arr{4 5.6}{4 4.4}
\arr{4 3.6}{4 2.4}
\arr{4 1.6}{4 .4}
\arr{4.4 4}{5.6 4}
\arr{4 9.6}{4 8.4}
\arr{2.4 8}{3.6 8}
\arr{2 7.6}{2 6.4}
\arr{.4 6}{1.6 6}
\arr{0 5.6}{0 4.4}
\put{$\ssize 3''$} at  -.5 4
\put{$\ssize 2''$} at  -.5 2
\put{$\ssize 4'$} at  1.5 6.5

\put{$\ssize \bar3$} at 6.5 4
\put{$\ssize 1$} at 3.5 0
\put{$\ssize 2$} at 3.5 1.5
\put{$\ssize 3$} at 3.5 3.5
\put{$\ssize 4$} at 3.5 6.9
\put{$\ssize 5$} at 3.5 8.5

\put{$\ssize 4''$} at -.5 6
\put{$\ssize 5'$} at 1.5 8
\put{$\ssize 6$} at 3.5 10

\setdots<2pt>
\plot .6 3.7  3.4 2.3 /
\plot 4.2 7.9  4.4 7.8  4.5 7.6  4.5 0.4  4.4 0.2  4.2 0.1 /
\plot 4.2 9.9  4.6 9.8  4.7 9.6  4.7 2.4  4.6 2.2  4.2 2.1 /
\plot .1 4.2  .2 4.4  .4 4.5  5.6 4.5  5.8 4.4  5.9 4.2 /
\plot 2.1 5.8  2.2 5.6  2.4 5.5  5.7 5.5  5.9 5.4  6 5.2  6 4.4 /
\plot .6 5.7  3.4 4.3 /
\plot 2.3 7.7  3.7 6.3 /
\endpicture}
$$

\begin{theorem}[{\cite[Theorem 3.4]{diss}}]
The algebra $A$ is a tubular algebra of type $\mathbb T=(4,4,2)$.
\qed
\end{theorem}

In particular, the indecomposable $A$-modules occur in the following 
families:  A preprojective component $\mathcal P$;
a nonstable family of tubes $\mathcal T_0$ with three projective objects;
for each $\gamma\in\mathbb Q^+$ a stable family $\mathcal T_\gamma$ 
of tubes of type $\mathbb T$;  a nonstable family $\mathcal T_\infty$
with two injective objects; and a preinjective component.

\medskip
An $A$-module $N$ can be considered an object, say $M$, 
in $\mathcal C(\widetilde 4)$ as follows.  
$$M_i=\left\{\begin{array}{cc}N_i & 1\leq i\leq 5\\ 0 &\text{otherwise}
\end{array}\right.\quad
M'_i=\left\{\begin{array}{cc}N_i & i< 3 \\ 
   K & i=3 \\ N'_4 & i=4\\ 0 & i>4 \end{array}\right.
\quad M''_i=\left\{\begin{array}{cc}N_i & i<2 \\ N''_i & i=2,3\\ 0 & i>3
  \end{array}\right.$$
Here, $K$ is the kernel of the map $\pi: N_3\to \bar N_3$.  Note that if 
$N$ has no direct summand isomorphic to the projective $A$-module 
$P(\bar 3)$,
then the map $\pi$ is an epimorphism and $N$ can be recovered from $M$.

\medskip
Each family $\mathcal T_\gamma$ consists of three extended tubes 
of circumference 4, 4, and 2, respectively, and a one-parameter family
of homogeneous tubes.  If $\gamma=\frac ab\in\mathbb Q^+_0\cup\{\infty\}$
written in lowest terms, then the 
quasisimple module in each homogeneous tube has dimension type
$a\,{\mathbf h}[1]+b\,{\mathbf h}$.  Here, $\mathbf h$ and $\mathbf h[1]$
are the dimension type of quasisimple homogeneous objects in $\mathcal T_0$
and $\mathcal T_\infty$ (written as objects in $\mathcal S(\widetilde 4)$):
$$ {\mathbf h} = \begin{smallmatrix} 
&&0 \\ &0&1\\ 0&1&2\\ 1&2&2\\ 1&1&1
\end{smallmatrix},\quad
{\mathbf h}[1] = \begin{smallmatrix} 
&&1 \\ &1&2\\ 1&2&2\\ 1&1&1\\ 0&0&0
\end{smallmatrix}$$

\subsection{Three inserted rays}\label{section-three-rays}
%-------------------------------
So far, none of the projective indecomposable objects 
in $\mathcal C(\widetilde 4)$ occurs as an $A$-module.
But it turns out that the radicals of the projective objects 
$P(4'')$, $P(5')$ and $P(6)$ occur in the component $\mathcal T_\infty$
of $A$-mod.  
One of the two tubes of circumference 4 contains $\rad P(4'')$ and $\rad P(6)$,
the other contains $\rad P(5')$.
The path algebra $A^+$ obtained by forming the corresponding
three one-point extensions is given by the quiver $Q_A^+$
with relations, as pictured above.

\medskip
The Auslander-Reiten quiver for $A^+$-mod has the following shape.
First, there are the preprojective component $\mathcal P$, 
the nonstable family of tubes $\mathcal T_0$, 
and the stable families of tubes $\mathcal T_\gamma$ 
of type $\mathbb T$ where $\gamma\in\mathbb Q^+$.
They all consist of $A$-modules.

\smallskip
The next family of tubes, $\mathcal T_\infty^+$, is obtained 
from $\mathcal T_\infty$ by inserting three rays at the modules 
$\rad P(4'')$, $\rad P(6)$ and $\rad P(5')$. 
Then there are further components which we will not specify.

\medskip
The indecomposable $A^+$-modules, with the exception of the projective
$P(\bar 3)$ are all objects in the category $\mathcal C(\widetilde 4)$;
the above identification of $A$-modules with no summand isomorphic to $P(\bar 3)$
as objects in $\mathcal C(\widetilde 4)$ can be adapted.
Moreover, the families $\mathcal P$ (with the exception of $P(\bar 3)$), 
$\mathcal T_0$, and $\mathcal T_\gamma$ for $0<\gamma<\infty$
consist of objects in the subcategory $\mathcal S(\widetilde 4)$. 
The simple object $S(2'')$ occurs on the mouth of a tube 
in $\mathcal T^+_\infty$; note that $S(2'')$ and all modules $X$ with
$\Hom(X,S(2''))\neq0$ are not in $\mathcal S(\widetilde 4)$ since the maps
$X_{2''}\to X_{2'}$ are not monic.
Within the mentioned components, such modules $X$ occur exactly on one coray
ending at the object $S(2'')$ in $\mathcal T^+_\infty$.
We write $\mathcal U=\mathcal T^+_\infty\cap
\mathcal S(\widetilde 4)$ and define 
$\mathcal D=\bigsqcup_{\gamma\in\mathbb Q^+}\mathcal T_\gamma\sqcup \mathcal U$.
Thus $\mathcal U$ consists of the following components:
A one parameter family of homogeneous tubes, each quasisimple module
has dimension type $\mathbf h[1]$; 
a stable extended tube of
circumference 2; an extended tube of circumference 4 which contains the
projective-injective object $P(5')$; and an extended tube of 
circumference 4 which contains the two projective-injective objects
$P(4'')$ and $P(6)$.

\begin{theorem}[{\cite[Proposition 3.12]{diss}}]
The modules in $\mathcal D$
form the fundamental domain for the shift $[1]$ in $\mathcal S(\widetilde 4)$.
\qed
\end{theorem}

We list here the dimension vectors on the modules on the three rays
starting at the three projective-injective modules $P(4'')$, $P(5')$, $P(6)$
since they are the only indecomposable objects in $\mathcal D$ which 
are not $A$-modules.

$$
\begin{smallmatrix} 
&&0 \\ &0&0\\ 1&1&1\\ 1&1&1\\ 1&1&1\\ 1&1&1
\end{smallmatrix} 
\to 
\begin{smallmatrix} 
&&0 \\ &0&0\\ 1&1&1\\ 1&1&1\\ 1&1&1\\ 0&0&0
\end{smallmatrix} 
\to 
\begin{smallmatrix} 
&&0 \\ &0&1\\ 1&1&2\\ 1&1&2\\ 1&1&1\\ 0&0&0
\end{smallmatrix} 
\to 
\begin{smallmatrix} 
&&0 \\ &0&1\\ 1&1&2\\ 1&2&2\\ 1&1&1\\ 0&0&0
\end{smallmatrix} 
\to 
\begin{smallmatrix} 
&&0 \\ &0&1\\ 1&2&3\\ 2&3&3\\ 2&2&2\\ 1&1&1
\end{smallmatrix} 
\to \cdots $$

$$
\begin{smallmatrix} 
&&0 \\ &1&1\\ 0&1&1\\ 0&1&1\\ 0&1&1\\ 0&0&0
\end{smallmatrix} 
\to 
\begin{smallmatrix} 
&&0 \\ &1&1\\ 0&1&1\\ 0&1&1\\ 1&1&1\\ 0&0&0
\end{smallmatrix} 
\to 
\begin{smallmatrix} 
&&0 \\ &1&1\\ 0&1&1\\ 1&2&2\\ 1&1&1\\ 0&0&0
\end{smallmatrix} 
\to 
\begin{smallmatrix} 
&&0 \\ &1&1\\ 0&1&2\\ 1&2&2\\ 1&1&1\\ 0&0&0
\end{smallmatrix} 
\to 
\begin{smallmatrix} 
&&0 \\ &1&2\\ 0&2&3\\ 1&3&3\\ 1&2&2\\ 0&0&0
\end{smallmatrix} 
\to \cdots $$

$$
\begin{smallmatrix} 
&&1 \\ &0&1\\ 0&0&1\\ 0&0&1\\ 0&0&0\\ 0&0&0
\end{smallmatrix} 
\to 
\begin{smallmatrix} 
&&1 \\ &0&1\\ 0&0&1\\ 0&1&1\\ 0&0&0\\ 0&0&0
\end{smallmatrix} 
\to 
\begin{smallmatrix} 
&&1 \\ &0&1\\ 0&1&2\\ 1&2&2\\ 1&1&1\\ 1&1&1
\end{smallmatrix} 
\to 
\begin{smallmatrix} 
&&1 \\ &0&1\\ 0&1&2\\ 1&2&2\\ 1&1&1\\ 0&0&0
\end{smallmatrix} 
\to 
\begin{smallmatrix} 
&&1 \\ &0&2\\ 0&1&3\\ 1&2&3\\ 1&1&1\\ 0&0&0
\end{smallmatrix} 
\to \cdots $$

Note that after every 4 steps, the dimension vector increases by
$\mathbf h[1]$.

\subsection{The category $\mathcal S(4)$}
%----------------------------------------
The categories $\mathcal S(\widetilde 4)$ and $\mathcal S(4)$ are related by an 
important functor, the the so-called covering functor
$$\mathcal S(\widetilde 4)\to \mathcal S(4)$$
It assigns to an object $M\in\mathcal S(\widetilde 4)$ which is given 
as a representation $M$ of $\widetilde Q$ the following quadruple 
$(V,T,U_1,U_2)\in \mathcal S(4)$:
$$V=\bigoplus_{i\in\mathbb Z}M_i,\quad  U_2=\bigoplus_{i\in\mathbb Z}M_i',
\quad U_1=\bigoplus_{i\in\mathbb Z}M_i''$$
The embedding $U_1\to U_2$ is given by the maps $M_\beta'$, the embedding 
$U_2\to V$ by the maps $M_\beta$.  The operation of $T$ on $V$ is given by the maps $\alpha$.

\begin{theorem}[{\cite[below Lemma 3.13]{diss}}]
The covering functor $\mathcal S(\widetilde 4)\to \mathcal S(4)$ is dense,
it induces a bijection between the $\mathbb Z$-orbits of isomorphism classes of 
indecomposable objects in $\mathcal S(\widetilde 4)$ and the isomorphism classes of 
indecomposable objects in $\mathcal S(4)$. 
\qed
\end{theorem}

Combining the two above theorems, the covering functor induces a bijection between 
the objects in $\mathcal D$ and the isomorphism classes of indecomposable objects in 
$\mathcal S(4)$.  In particular, the indecomposable objects in $\mathcal S(4)$ 
are either $A$-modules in one of the families of tubes 
$\big(\bigsqcup_{\gamma\in\mathbb Q^+}\mathcal T_\gamma\big)\sqcup\big(\mathcal T_\infty\cap
\mathcal S(\widetilde 4)\big)$
or lie in one of the three rays in $\mathcal U$ 
starting at a projective-injective object.

\section{The possible dimension types}
%---------------------------------------

\subsection{A symmetry result for dimension types}
%------------------------------------------------
\label{section-symmetry}
When considering the diagrams under Theorem~\ref{theorem-dimension-triples}, 
we observe that the set of dimension types, 
as a set of points in $\mathbb R^3$,
is symmetric with respect to reflection on the 
plane $x=z$.  The following lemma confirms this observation:

\begin{lemma}
For each $n$, 
the category $\mathcal S(n)$ has a self duality $R$ which induces an
involution on the set of dimension types given by 
$$\dim RM\;=\;(z,y,x)$$
if $M\in\mathcal S(n)$ has dimension type $\dim M=(x,y,z)$.
\end{lemma}

\begin{proof}
Let $\mathcal F(n)$ be the category of all composable epimorphisms of 
$k[T]/(T^n)$-modules.  Then $\mathcal F(n)$ and $\mathcal S(n)$ are
both equivalent and dual:

\smallskip
The reflection duality
$$E:\quad (U_1\subset U_2\subset V)\;\mapsto\; (V\onto V/U_1\onto V/U_2)$$
studied in \cite[Section  5.2]{simson1} actually is an equivalence of 
categories.  Combining $E$ with the vector space duality $D=\Hom_k(-,k)$
yields a self duality $R=DE$ on $\mathcal S(n)$:
$$R:\quad (U_1\subset U_2\subset V)\;\mapsto\; (D(V/U_2)\subset D(V/U_1)\subset D(V))$$
If $x=\dim U_1$, $y=\dim U_2/U_1$, $z=\dim V/U_2$, then 
$\dim D(V/U_2)=z$, $\dim D(V/U_1)=y+z$, $\dim D(V)=x+y+z$, so 
$R(U_1\subset U_2\subset V)$ has dimension type $(z,y,x)$.
\end{proof}

\subsection{From dimension vectors to dimension types...}
%----------------------------------------------------------
It is the aim of this section to show the first part of
Theorem~\ref{theorem-dimension-triples}:

\begin{proposition}\label{proposition-dimension-triples}
Let $M:(V,T,U_1,U_2)$ be an indecomposable system in $\mathcal S(4)$.
Then the dimension type
$$\dim M=(\dim U_1, \dim U_2/U_1, \dim V/U_2)$$
is in the marked region in the hexagonal diagrams under 
Theorem~\ref{theorem-dimension-triples}.
\end{proposition}

The covering functor $\pi:\mathcal S(\tilde 4)\to \mathcal S(4)$
induces a $\mathbb Z$-linear map on the Grothendieck groups
$$\pi:\;K_0(\mod \tilde Q)\to K_0(\mod Q),\; \textstyle
\bd\mapsto(\sum_i d_{i''}, \sum_i(d_{i'}-d_{i''}),\sum_i(d_{i'}-d_i)),
$$
where $\widetilde Q$, $Q$ are the quivers introduced in 
Section~\ref{section-quiver-representations}.
Clearly, for an object $X\in\mathcal S(\tilde 4)$, the formula
$\pi(\dim X)=\dim(\pi(X))$ holds.

\begin{proof}
In the proof, we identify for each dimension type $(x,y,z)$ 
in the marked region a corresponding
object $X\in\mathcal S(\widetilde 4)$ such that $\pi(\dim X)=(x,y,z)$.

\medskip
We may assume that $M=\pi(X)$ where $X$ is in the fundamental domain
$\mathcal D$ for $\mathcal S(\tilde 4)$.  So $X$ is either an $A$-module
or lies in one of the three inserted rays.

\medskip
We first deal with the modules on the inserted rays.  
By applying $\pi$ to the dimension vectors listed at the end 
of Section~\ref{section-three-rays}, 
the modules on the rays starting at $P(4'')$, $P(5')$, $P(6)$
give rise to the following dimension types.
\begin{eqnarray*}
& (400) \to (300)\to (303)\to (312)\to (622)\to \cdots\\
& (040) \to (130)\to (250)\to (251)\to (262)\to \cdots \\
& (004) \to (013)\to (323)\to (223)\to (226)\to \cdots
\end{eqnarray*}
As mentioned above, after every 4 steps the dimension vectors on the rays 
increase by $\bh[1]$.  Thus, after every 4 steps, the dimension type
increase by $\pi(\bh[1])=(2,2,2)$.
Thus, by checking the first four dimension types in each sequence,
we easily verify that each dimension type of a module 
on one of the rays is in the marked region in the hexagonal diagrams.

\medskip
The remaining indecomposable modules in $\mathcal D$ are all $A$-modules.
As the category of modules over the tubular algebra $A$ is controlled
by an integral quadratic form $\chi_A$, the dimension vector $\bd$ of an
indecomposable $A$-module is either a positive root 
(that is, $\chi_A(\bd)=1$) or a positive radical
vector for $\chi_A$ (that is, $\chi_A(\bd)=0$).  
It turns out that the radical of the tubular algebra $A$ 
is the subgroup of $K_0(A)$ 
generated by $\bh$ and $\bh[1]$ (see \cite[5.1 (1)]{ringel}).

\medskip
If the dimension vector $\bd$ of the $A$-module $X$ is in the radical,
say $\bd=a_0\bh+a_\infty\bh[1]$, then the dimension type for
$M=\pi(X)$ is $(a_0+a_\infty)\cdot(2,2,2)$.  

\medskip
In order to deal with all roots, we use the fact 
that for a radical vector $\mathbf u$
and a dimension vector $\bd$, the quadratic form satisfies
$$\chi_A(\bd+\mathbf u) = \chi_A(\bd).$$
This property allows us to reduce a root $\bd$ for $\chi_A$ to a root 
for $\chi_{A'}$ (see below) as follows.
Put $\bd'=\bd-\bd_{\bar3}\bh-\bd_5\bh[1]$.
Since $\bh$ and $\bh[1]$ are in the radical, we have
$$1\;=\;\chi(\bd)\;=\;\chi(\bd').$$
Since $\bh$ has entry 1 in position $\bar3$ and $\bh[1]$ has entry
1 in position 5, we have $\bd'_5=\bd'_{\bar3}=0$, so $\bd'$ is a root for the
quadratic form of the algebra $A'$ given by the following quiver with
relations (which is obtained from the quiver for $A$ by deleting the 
points $\bar 3$ and $5$, and the corresponding arrows and relations).

$$
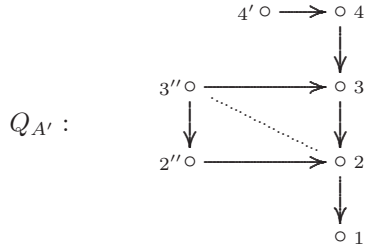
\hbox{\beginpicture
\setcoordinatesystem units <0.5cm,0.5cm>
%==========================================
\put{} at 0 0
\put{} at 4 6
\put{$Q_{A'}:$} at -4 3
\multiput{$\circ$} at 0 2  0 4  2 6  4 0  4 2  4 4  4 6 /
\arr{2.4 6}{3.6 6}
\arr{0.4 4}{3.6 4}
\arr{0.4 2}{3.6 2}
\arr{0 3.6}{0 2.4}
\arr{4 5.6}{4 4.4}
\arr{4 3.6}{4 2.4}
\arr{4 1.6}{4 .4}

\put{$\ssize 3''$} at  -.5 4
\put{$\ssize 2''$} at  -.5 2
\put{$\ssize 4'$} at  1.5 6

\put{$\ssize 1$} at 4.5 0
\put{$\ssize 2$} at 4.5 2
\put{$\ssize 3$} at 4.5 4
\put{$\ssize 4$} at 4.5 6

\setdots<2pt>
\plot .6 3.7  3.4 2.3 /
\endpicture}
$$
The algebra is tilted of type $\mathbb E_7$; the roots for $\chi_{A'}$
are obtained from the roots of the Dynkin diagram $\mathbb E_7$ as follows.
Suppose $\left(\begin{smallmatrix}&&b''\\a&b&g&c&d&d'\\ \end{smallmatrix}\right)$
is a root for $\mathbb E_7$, then the corresponding root for 
$\chi_{A'}$ is $\left(\begin{smallmatrix}&d'&d\\c''&&c\\b''&&b\\&&a\end{smallmatrix}\right)$
where $c''=b''+c-g$, and gives hence rise to the dimension type
\begin{eqnarray*}(x,y,z)& = & (a+b''+c'',b-b''+c-c''+d',d-d') \\
         &= &(a+2b''+c-g,b-2b''+g+d',d-d')\end{eqnarray*}
In the following diagram, we take each of the 63 positive roots of 
$\mathbb E_7$, for example from \cite[Planche VI]{bourbaki}, 
and represent the corresponding dimension type in the
hexagonal diagram by putting the number $\min(x,y,z)$ at position $(x,y,z)$.
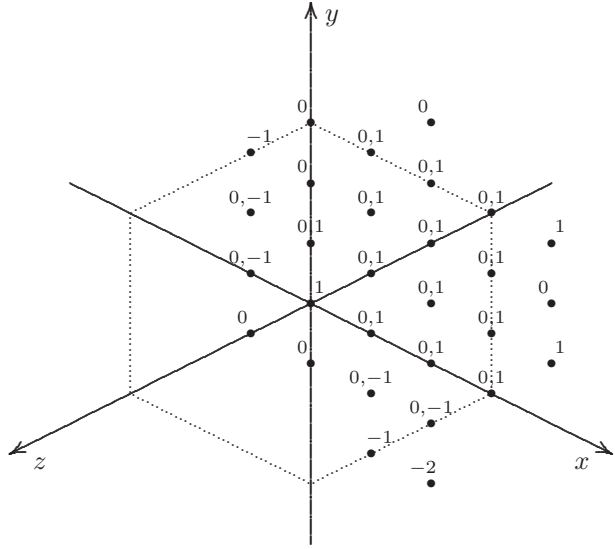
$$
\beginpicture\setcoordinatesystem units <4mm,4mm>
\put {} at -10 -8
\put {} at 10 10 
\arr{0 -8}{0 10}
\arr{8 4}{-10 -5}
\arr{-8 4}{10 -5}
\put{$x$} at 9 -5.3
\put{$y$} at .7 9.5
\put{$z$} at -9 -5.3
\setdots<2pt>
\plot 0 6  -6 3  -6 -3  0 -6  6 -3  6 3  0 6 /
\multiput{$\ssize\bullet$} at -2 -1  -2 1  -2 3  -2 5  0 -2  0 0  0 2  0 4  0 6
            2 -5  2 -3  2 -1  2 1  2 3  2 5  4 -6  4 -4  4 -2  4 0  4 2  4 4  4 6 
            6 -3  6 -1  6 1  6 3  8 -2  8 0  8 2 /
\multiput{$\ssize0\phantom{,1}$} at -2 -.5  0 6.5  0 4.5  0 -1.5  4 6.5  8 0.5 /
\multiput{$\ssize0,1$} at 0 2.5  2 -.5  2 1.5  2 3.5  2 5.5  4 -1.5  4 0.5  4 2.5  4 4.5
                           6 -2.5  6 -.5  6 1.5  6 3.5 /
\multiput{$\ssize\phantom{0,}1$} at 0 0.5  8 -1.5  8 2.5 /
\multiput{$\ssize\phantom{0,}-1$} at -2 5.5 2 -4.5 /
\multiput{$\ssize0,-1$} at -2 1.5  -2 3.5  2 -2.5  4 -3.5  /
\multiput{$\ssize-2\phantom{,0}$} at 4 -5.5 / 
\endpicture
$$
We note that all roots fall into the marked region in the hexagonal 
diagrams following Theorem~\ref{theorem-dimension-triples}.
The negative roots for $\mathbb E_7$ give rise to the negative 
dimension types; it is straightforward to verify that they, too, 
fall in the marked region in the hexagonal diagrams.
\end{proof}

\medskip
We observe that none of the integer multiples of $(2,2,2)$ is
realized as the dimension type of a root.

\begin{corollary}
Suppose the indecomposable object $X\in\mathcal S(\widetilde4)$ has
dimension vector $\bd$ and dimension type $\pi(\bd)$. 
Then $\bd$ is a linear combination of $\bh$ and $\bh[1]$ if and only
if $\pi(\bd)$ is an integer multiple of $(2,2,2)$. 
In particular, if $X$ is on one of the inserted rays, or if $\bd$
is a root, then $\pi(\bd)$ is not an integer multiple of $(2,2,2)$. 
\end{corollary}

\begin{proof}
We may assume that $X$ is in the fundamental domain $\mathcal D$, 
so either $X$ is an $A$-module or $X$ is on one of the three inserted
rays.

\smallskip
We have seen that for the modules in the inserted rays,
none of the dimension vectors is a linear combination of 
$\bh$ and $\bh[1]$, and the corresponding dimension types
are not multiples of $(2,2,2)$. 

\smallskip
It remains to deal with $A$-modules. 
Obviously, the radical generators
$\bh$ and $\bh[1]$ for $\chi_A$ and all their linear combinations 
are mapped to multiples of $(2,2,2)$.
Conversely, we have seen 
in the proof of Proposition~\ref{proposition-dimension-triples}
that none of the roots is mapped to a multiple of $(2,2,2)$. 
\end{proof}

\subsection{...and back to dimension vectors}
%--------------------------------------------

We show that each triple $(x,y,z)$ which satisfies the conditions 
in Theorem~\ref{theorem-dimension-triples} can be realized as the dimension 
type of an indecomposable object in $\mathcal S(4)$.

\smallskip
First note that dimension types of the form
$(4,0,0)+\mathbb N(2,2,2)$, $(0,4,0)+\mathbb N(2,2,2)$,
$(0,0,4)+\mathbb N(2,2,2)$ and $(3,0,3)+\mathbb N(2,2,2)$
occur on the three inserted rays (the corresponding objects in 
$\mathcal S(4)$ are pictured in \cite[(3.5)]{diss}).

\bigskip
For the remaining triples, it suffices to assume that they have the form 
$(x,y,z)$ where $x\leq z$.  Using reflection duality 
(Section~\ref{section-symmetry}) 
we can obtain the corresponding triples with $x\geq z$.

\medskip
We realize those triples as modules over algebras
$A^{(12)}$ or $A^{(3)}$ which are tame concealed or tame domestic,
respectively.

\smallskip
Consider the 
path algebras given by the following two quivers with relations:
$$
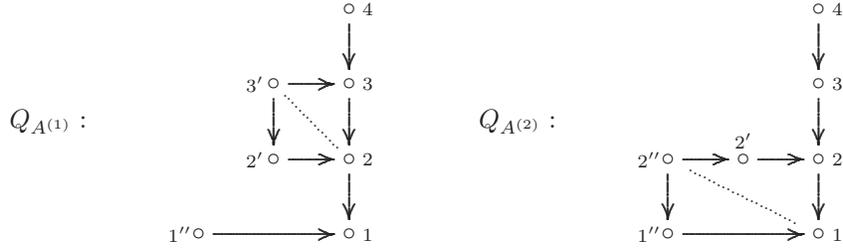
\hbox{\beginpicture
\setcoordinatesystem units <0.5cm,0.5cm>
%==========================================
\put{} at 0 0
\put{} at 4 6
\put{$Q_{A^{(1)}}:$} at -4 3
\multiput{$\circ$} at 0 0  2 2  2 4  4 0  4 2  4 4  4 6 /
\arr{2.4 4}{3.6 4}
\arr{2.4 2}{3.6 2}
\arr{0.4 0}{3.6 0}
\arr{2 3.6}{2 2.4}
\arr{4 5.6}{4 4.4}
\arr{4 3.6}{4 2.4}
\arr{4 1.6}{4 .4}

\put{$\ssize 1''$} at  -.5 0
\put{$\ssize 2'$} at  1.5 2
\put{$\ssize 3'$} at  1.5 4

\put{$\ssize 1$} at 4.5 0
\put{$\ssize 2$} at 4.5 2
\put{$\ssize 3$} at 4.5 4
\put{$\ssize 4$} at 4.5 6

\setdots<2pt>
\plot 2.3 3.7  3.7 2.3 /
\endpicture}
\qquad\qquad
%
%  A^(2)
%
\hbox{\beginpicture
\setcoordinatesystem units <0.5cm,0.5cm>
%==========================================
\put{} at 0 0
\put{} at 4 6
\put{$Q_{A^{(2)}}:$} at -4 3
\multiput{$\circ$} at 0 0  0 2  2 2  4 0  4 2  4 4  4 6 /
\arr{0.4 2}{1.6 2}
\arr{2.4 2}{3.6 2}
\arr{0.4 0}{3.6 0}
\arr{0 1.6}{0 0.4}
\arr{4 5.6}{4 4.4}
\arr{4 3.6}{4 2.4}
\arr{4 1.6}{4 .4}

\put{$\ssize 1''$} at  -.5 0
\put{$\ssize 2''$} at  -.5 2
\put{$\ssize 2'$} at  2 2.5

\put{$\ssize 1$} at 4.5 0
\put{$\ssize 2$} at 4.5 2
\put{$\ssize 3$} at 4.5 4
\put{$\ssize 4$} at 4.5 6

\setdots<2pt>
\plot .6 1.7  3.4 .3 /
\endpicture}
$$

Both algebras have finite representation type.  A representations $X$ for
which the horizontal maps are monomorphisms can be considered as
an object $V$ in $\mathcal S(\widetilde 4)$, note that we put 
$V_1'=V_1=X_1$.

\smallskip
For each such indecomposable representation $X$ we indicate the 
corresponding dimension type $(x,y,z)=\dim(\pi(V))$ in the
hexagonal diagram below by putting the number $\min(x,y,z)$ at 
position $(x,y,z)$. 
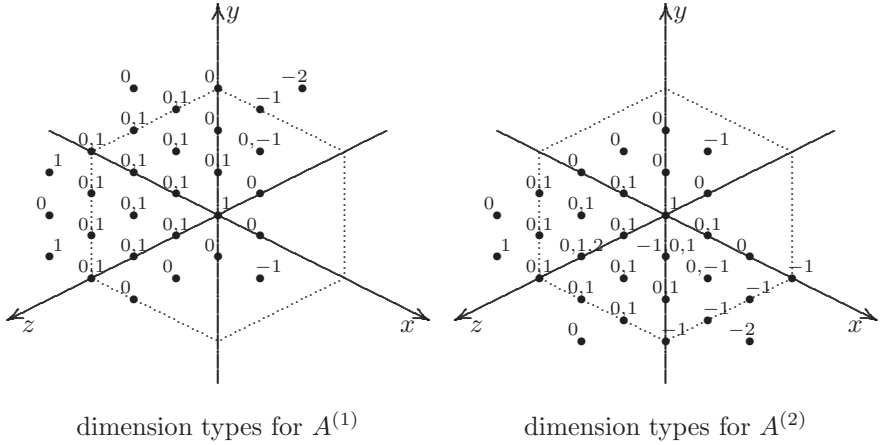
$$
\beginpicture\setcoordinatesystem units <2.8mm,2.8mm>
\put {} at -10 -10
\put {} at 10 10 
\put{dimension types for $A^{(1)}$} at 0 -10
\arr{0 -8}{0 10}
\arr{8 4}{-10 -5}
\arr{-8 4}{10 -5}
\put{$x$} at 9 -5.3
\put{$y$} at .7 9.5
\put{$z$} at -9 -5.3
\setdots<2pt>
\plot 0 6  -6 3  -6 -3  0 -6  6 -3  6 3  0 6 /
\multiput{$\ssize\bullet$} at -8 -2  -8 0  -8 2
                              -6 -3  -6 -1  -6 1  -6 3
                              -4 -4  -4 -2  -4 0  -4 2  -4 4  -4 6
                              -2 -3  -2 -1  -2 1  -2 3  -2 5
                              0 -2  0 0  0 2  0 4  0 6
                              2 -3  2 -1  2 1  2 3  2 5
                              4 6 /
\multiput{$\ssize0\phantom{,1}$} at -8 0.5  -4 -3.5  -4 6.5  -2 -2.5  0 -1.5  0 4.5  
                                  0 6.5  2 -.5  2 1.5 /
\multiput{$\ssize0,1$} at -6 -2.5  -6 -.5  -6 1.5  -6 3.5  -4 -1.5  -4 .5  -4 2.5  -4 4.5
                        -2 -.5  -2 1.5  -2 3.5  -2 5.5  0 2.5 /
\multiput{$\ssize\phantom{0,}1$} at -8 -1.5  -8 2.5  0 .5  /
\multiput{$\ssize\phantom{0,}-1$} at 2 -2.5  2 5.5 /
\multiput{$\ssize0,-1$} at 2 3.5 /
\multiput{$\ssize-2\phantom{,0}$} at 4 6.5 / 
\endpicture
\quad
%
%  same for A^(2)
%
\beginpicture\setcoordinatesystem units <2.8mm,2.8mm>
\put {} at -10 -10
\put {} at 10 10 
\put{dimension types for $A^{(2)}$} at 0 -10
\arr{0 -8}{0 10}
\arr{8 4}{-10 -5}
\arr{-8 4}{10 -5}
\put{$x$} at 9 -5.3
\put{$y$} at .7 9.5
\put{$z$} at -9 -5.3
\setdots<2pt>
\plot 0 6  -6 3  -6 -3  0 -6  6 -3  6 3  0 6 /
\multiput{$\ssize\bullet$} at -8 -2  -8 0  
                              -6 -3  -6 -1  -6 1 
                              -4 -6  -4 -4  -4 -2  -4 0  -4 2  
                              -2 -5  -2 -3  -2 -1  -2 1  -2 3 
                              0 -6  0 -4  0 -2  0 0  0 2  0 4
                              2 -5  2 -3  2 -1  2 1  2 3 
                              4 -6  4 -4  4 -2  
                              6 -3 /
\multiput{$\ssize0\phantom{,1}$} at -8 0.5  -4 -5.5  -4 2.5  -2 3.5  0 2.5  0 4.5  
                                  2 1.5  4 -1.5 /
\multiput{$\ssize0,1$} at -6 -2.5  -6 -.5  -6 1.5  -4 -3.5  -4 .5  
                        -2 -4.5  -2 -2.5  -2 -.5  -2 1.5  0 -3.5  2 -.5 /
\multiput{$\ssize\phantom{0,}1$} at -8 -1.5  0 .5  /
\multiput{$\ssize\phantom{0,}-1$} at 0 -5.5  2 -4.5  2 3.5  4 -3.5  6 -2.5 /
\multiput{$\ssize0,-1$} at 2 -2.5 /
\multiput{$\ssize-2\phantom{,0}$} at 4 -5.5 / 
\multiput{$\ssize-1,0,1$} at 0 -1.5 /
\multiput{$\ssize 0,1,2$} at -4 -1.5 /
\endpicture
$$

We verify that each triple $(x,y,z)$ in 
Theorem~\ref{theorem-dimension-triples} with $x\leq z$ and $\min(x,y,z)\in\{0,1\}$
and different from $(4,0,0)$, $(0,4,0)$, $(0,0,4)$, $(3,0,3)$
can  be realized as the dimension type of an embedding $\pi(X)$ 
where $X$ is an indecomposable 
module over either $A^{(1)}$ or $A^{(2)}$ --- with one additional exception:  
The triple $(4,1,4)$ 
cannot be realized in this way and we will need to deal with it later.

\smallskip
The algebra $A^{(12)}$ is given by the quiver with relations,
$$
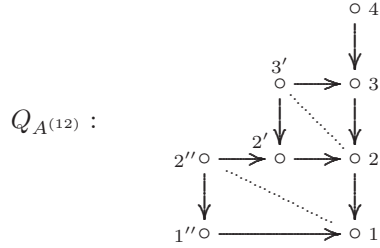
\hbox{\beginpicture
\setcoordinatesystem units <0.5cm,0.5cm>
%==========================================
\put{} at 0 0
\put{} at 4 6
\put{$Q_{A^{(12)}}:$} at -4 3
\multiput{$\circ$} at 0 0  0 2  2 2  2 4  4 0  4 2  4 4  4 6 /
\arr{2.4 4}{3.6 4}
\arr{0.4 2}{1.6 2}
\arr{2.4 2}{3.6 2}
\arr{0.4 0}{3.6 0}
\arr{0 1.6}{0 0.4}
\arr{2 3.6}{2 2.4}
\arr{4 5.6}{4 4.4}
\arr{4 3.6}{4 2.4}
\arr{4 1.6}{4 .4}

\put{$\ssize 1''$} at  -.5 0
\put{$\ssize 2''$} at  -.5 2
\put{$\ssize 2'$} at  1.5 2.5
\put{$\ssize 3'$} at 2 4.5

\put{$\ssize 1$} at 4.5 0
\put{$\ssize 2$} at 4.5 2
\put{$\ssize 3$} at 4.5 4
\put{$\ssize 4$} at 4.5 6

\setdots<2pt>
\plot .6 1.7  3.4 .3 /
\plot 2.3 3.7  3.7 2.3 /
\endpicture}
$$
it is a tame concealed algebra of type ${\widetilde{\mathbb E}}_7$, 
its radical vector is $\bh$.
Thus, if a homogeneous module $H$ has dimension vector $s\cdot \bh$,
and corresponds to an object $V\in\mathcal S(\widetilde 4)$,
then $\pi(V)$ has dimension type  $s\cdot(2,2,2)$. 

\smallskip
Let $(x,y,z)$ be a triple satisfying the condition in the theorem and such 
that $x\leq z$.  We assume that $(x,y,z)$ is not one of $(4,0,0)$, $(0,4,0)$,
$(0,0,4)$, $(3,0,3)$, $(4,1,4)$, up to a multiple of $(2,2,2)$.
Then $(x,y,z)=(x',y',z')+s\cdot(2,2,2)$ where $s$ is a non-negative number and
where the triple $(x',y',z')$ satisfies $\min(x',y',z')\in\{0,1\}$ and occurs in one of the
hexagonal diagrams above. Let $H$ be a homogeneous $A^{(12)}$-module of dimension
vector $s\cdot \bh$, and let $X$ be an indecomposable $A^{(1)}$- 
or $A^{(2)}$-module 
such that $\pi(X)$ is an embedding of  dimension type $(x',y',z')$. 
According to \cite[Lemma (3.1.1)]{rs}, there exists an extension $N$ of $X$ by $H$ or
of $H$ by $X$.  Then $N$ has the property that $\pi(N)$ is an embedding with 
$$\dim\pi(N)=(x',y',z')+s\cdot(2,2,2)=(x,y,z).$$

\bigskip
It remains to deal with the dimension types of the form $(4,1,4)+\mathbb Z(2,2,2)$.
Consider the algebra $A^{(3)}$:
$$
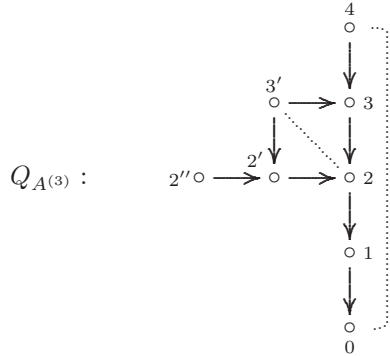
\hbox{\beginpicture
\setcoordinatesystem units <0.5cm,0.5cm>
%==========================================
\put{} at 0 -2
\put{} at 4 6
\put{$Q_{A^{(3)}}:$} at -4 2
\multiput{$\circ$} at  0 2  2 2  2 4  4 -2  4 0  4 2  4 4  4 6 /
\arr{2.4 4}{3.6 4}
\arr{0.4 2}{1.6 2}
\arr{2.4 2}{3.6 2}
\arr{2 3.6}{2 2.4}
\arr{4 5.6}{4 4.4}
\arr{4 3.6}{4 2.4}
\arr{4 1.6}{4 .4}
\arr{4 -.4}{4 -1.6}

\put{$\ssize 2''$} at  -.5 2
\put{$\ssize 2'$} at  1.5 2.5
\put{$\ssize 3'$} at 2 4.5

\put{$\ssize 0$} at 4 -2.5
\put{$\ssize 1$} at 4.5 0
\put{$\ssize 2$} at 4.5 2
\put{$\ssize 3$} at 4.5 4
\put{$\ssize 4$} at 4 6.5

\setdots<2pt>
\plot 2.3 3.7  3.7 2.3 /
\plot 4.6 6  4.8 6  5 5.9  5 -1.9  4.8 -2  4.6 -2 /
\endpicture}
$$
This algebra is tame domestic with radical generated by $\bh[1]$;
we verify that there is an indecomposable $A^{(3)}$-module $M$ of dimension vector $\bd$;
this module is preprojective.
$$
\bd=\begin{smallmatrix} &&1\\ &0&2\\ 1&2&3\\ &&2\\ &&1\end{smallmatrix}
$$
It turns out that the indecomposable $A^{(3)}$-modules of dimension vectors
$\be_{2''}$, $\be_{2'}$, and $\be_{2'}+\be_{3'}$ are regular or preinjective,
and hence $\pi(M)$ is an embedding. Let $c$ be the Coxeter transformation.
One verifies that $c^4(\bd)=\bd+\bh^3$, so for each natural number $s$
there is an indecomposable preprojective module $M_s$ of dimension vector 
$\bd+s\bh[1]$.  It follows that $\pi(M_s)$ is an embedding of dimension type
$(4,1,4)+s\cdot(2,2,2)$.

\smallskip
This finishes the proof of Theorem~\ref{theorem-dimension-triples}.

\section{The hexagonal pictures in the finite cases}\label{section-finite}
%===================================================

For $n<4$, each of the categories $\mathcal S(n)$ is of finite type.  
All the indecomposable objects have been determined explicitely in \cite[Section 3.1]{diss}.
Note that the dimension types there are given in the form $(\dim U_1,\dim U_2, \dim V)=(x,x+y,x+y+z)$.

\subsection{$n=1$}
%-----------------

In this case, the only indecomposables are the projective objects with
dimension types $(0,0,1)$, $(0,1,0)$, $(1,0,0)$.
$$
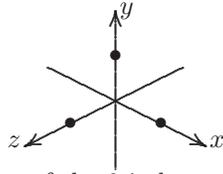
\beginpicture
  \setcoordinatesystem units <.3cm,.3cm>
  \put{} at -3 -2
  \put{} at 4 2
  \arr{0 -3}{0 4}
  \arr{3 1.5}{-4 -2}
  \arr{-3 1.5}{4 -2}
  \put{$x$} at 4.5 -1.8
  \put{$y$} at .5 4
  \put{$z$} at -4.5 -1.8
  \multiput{$\bullet$} at 0 2  -2 -1  2 -1 /
\endpicture
$$

\subsection{$n=2$}
%-----------------

The dimension types of the 9 indecomposable objects in $\mathcal S(2)$ in three space and in the 
hexagonal picture.

$$
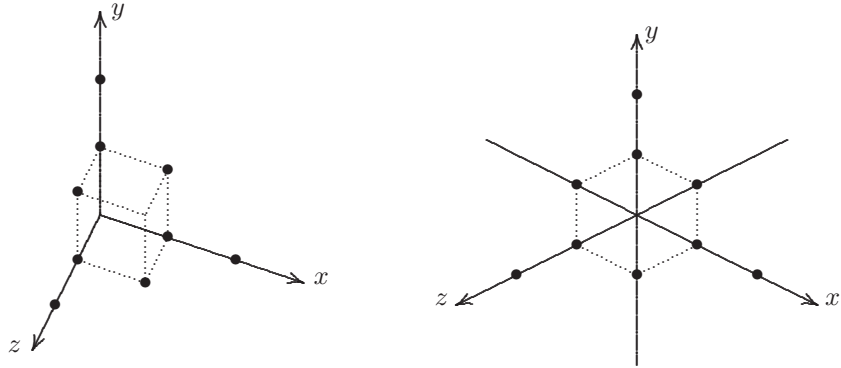
\beginpicture
  \setcoordinatesystem units <.3cm, .3cm>
  \put{} at -4 -6 
  \put{} at 9 9
  \arr{0 0}{-3 -6}
  \arr{0 0}{9 -3}
  \arr{0 0}{0 9}
  \put{$x$} at 9.8 -2.8
  \put{$y$} at .8 9
  \put{$z$} at -3.8 -5.8
  \multiput{$\bullet$} at -2 -4  6 -2  0 6 
    0 3  3 2  3 -1  2 -3  -1 -2  -1 1  /
  \setdots<2pt>
  \plot 0 3  3 2  3 -1  2 -3  -1 -2  -1 1  0 3 /
  \plot -1 1  2 0  3 2 /
  \plot 2 0  2 -3 /
  \plot -1 -2  0 0  0 3 /
  \plot 0 0  3 -1 /
\endpicture
\qquad\qquad
\beginpicture
  \setcoordinatesystem units <.4cm,.4cm>
  \put{} at -6 -5
  \put{} at 6 6
  \arr{0 -5}{0 6}
  \arr{5 2.5}{-6 -3}
  \arr{-5 2.5}{6 -3}
  \put{$x$} at 6.5 -2.8
  \put{$y$} at .5 6
  \put{$z$} at -6.5 -2.8
  \multiput{$\bullet$} at 0 4  -4 -2  4 -2
    0 2  2 1  2 -1  0 -2  -2 -1  -2 1 /
  \setdots <2pt>
  \plot -2 1      0 2  2 1  2 -1  0 -2  -2 -1  -2 1 /
\endpicture
$$

\subsection{$n=3$}
%-----------------

The dimension types of the 27 indecomposable objects in $\mathcal S(3)$ in three space.
Note that the points $(2,1,2)$ and $(1, 2, 1)$ are missing.
$$
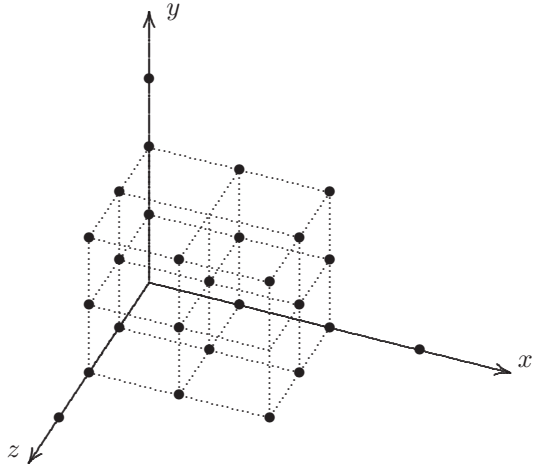
\beginpicture
  \setcoordinatesystem units <.4cm, .3cm>
  \put{} at -4 -8
  \put{} at 12 12
  \arr{0 0 }{12 -4}
  \arr{0 0}{0 12}
  \arr{0 0 }{-4 -8}
  \put{$x$} at 12.5 -3.5
  \put{$y$} at .8 12
  \put{$z$} at -4.5 -7.5
  \multiput{$\bullet$} at 0 9  9 -3  -3 -6 
    0 6  3 5  6 4  -1 4  0 3  -2 2  3 2  5 2  -1 1  1 1  6 1 
    2 0  4 0  -2 -1  3 -1  5 -1  -1 -2  1 -2  6 -2  2 -3  -2 -4  5 -4  1 -5  4 -6 /
\setdots <2pt>
  \plot 0 6  6 4 /
  \plot -1 4  5 2 /
  \plot -2 2  4 0 /
  \plot -2 -1 4 -3 /
  \plot -2 -4  4 -6 /
  \plot 0 6 -2 2 /
  \plot 3 5  1 1 /
  \plot 6 4  4 0 /
  \plot 6 1  4 -3 /
  \plot 6 -2  4 -6 /
  \plot -2 2  -2 -4 /
  \plot 1 1  1 -5 /
  \plot 4 0  4 -6 /
  \plot 5 2  5 -4 /
  \plot 6 4  6 -2 /
  \plot 0 3  -2 -1 /
  \plot 0 0  -2 -4 /
  \plot 3 2  1 -2 /
  \plot 3 -1  1 -5 /
  \plot 0 0  0 6 /
  \plot -1 -2  -1 4 /
  \plot 3 -1  3 5 /
  \plot 2 -3  2 3 /
  \plot 0 0  6 -2 /
  \plot -1 -2  5 -4 /
  \plot 0 3  6 1 /
  \plot -1 1  5 -1 /
\endpicture
$$

After rotation, we obtain the hexagonal picture.  Note the symmetry with respect to the plane 
$x=z$ given by the functor $R$ in Section~\ref{section-symmetry}.

$$
\beginpicture
  \setcoordinatesystem units <.4cm,.4cm>
  \put{} at -8 -7
  \put{} at 8 8
  \arr{0 -7}{0 8}
  \arr{7 3.5}{-8 -4}
  \arr{-7 3.5}{8 -4}
  \put{$x$} at 8.5 -3.8
  \put{$y$} at .5 8
  \put{$z$} at -8.5 -3.8
%  \put{$\bigcirc$} at 0 0 
  \multiput{$\bullet$} at 0 6  -6 -3  6 -3
    0 4  2 3  4 2  4 0  4 -2  2 -3  0 -4  -2 -3  -4 -2  -4 0  -4 2  -2 3
    0 2  2 1  2 -1  0 -2  -2 -1  -2 1  
    2.3 1.3  2.3 -.7  -1.7 -.7  -1.7 1.3  .3 .3  .6 .6 /
  \setdots <2pt>
  \plot 0 4  4 2  4 -2  0 -4  -4 -2  -4 2  0 4 /
\endpicture
$$

\end{document}